\documentclass[10pt,reqno]{amsart}

\usepackage{amsmath,amssymb,latexsym,amsthm}

\usepackage[unicode]{hyperref}

\theoremstyle{theorem}
\newtheorem{theorem}{Theorem}

\theoremstyle{theorem}
\newtheorem{lemma}{Lemma}

\theoremstyle{theorem}
\newtheorem{corollary}{Corollary}

\theoremstyle{theorem}
\newtheorem{proposition}{Proposition}

\theoremstyle{theorem}
\newtheorem{conjecture}{Conjecture}

\theoremstyle{definition}

\newtheorem*{remark}{Remark}

\DeclareMathOperator{\ord}{ord}

\begin{document}

\title[Certain equations of Erd\H{o}s--Moser type]{On the unsolvability of certain equations of Erd\H{o}s--Moser type}

\author{Ioulia N. Baoulina}

\address{Department of Mathematics, Moscow State Pedagogical University, Krasnoprudnaya str. 14, Moscow 107140, Russia}
\email{jbaulina@mail.ru}

\date{}

\maketitle

\thispagestyle{empty}

\section{Introduction}

For positive integers $k$ and $m$, let $S_k(m):=\sum_{j=1}^{m-1}j^k$ be the sum of $k$th powers of the first $m-1$ positive integers. In 2011, Kellner~\cite{K3} conjectured that for $m>3$, the ratio $S_k(m+1)/S_k(m)$ cannot be an integer. Since $S_k(m+1)=S_k(m)+m^k$, one can reformulate Kellner's conjecture as follows: for any positive integer $a$, the equation $aS_k(m)=m^k$ has no solutions $(m,k)$ with $m>3$. The special case $a=1$ of this conjecture is called the Erd\H{o}s--Moser conjecture. It was proposed around 1950 by Paul Erd\H{o}s in a letter to Leo Moser. Moser~\cite{Moser} proved that if $(m,k)$ is a solution of $S_k(m)=m^k$ with $m>3$, then $m>10^{10^6}$. As another supporting fact of the Erd\H{o}s--Moser conjecture, we mention the best known lower bound $m > 2.7139 \cdot 10^{\,1\,667\,658\,416}$, which is due to Gallot, Moree and Zudilin~\cite{GMZ}. 

While the Erd\H{o}s--Moser conjecture remains open, the unsolvability of the equation $aS_k(m)=m^k$ (which we call the Kellner--Erd\H{o}s--Moser equation) was recently~\cite{BM} established for many integers $a>1$. In particular, it was shown \cite[Theorem~1.5]{BM} that if $a$ is even or $a$ has a regular prime divisor or $2\le a\le 1500$, then the equation $aS_k(m)=m^k$  has no solutions $(m,k)$ with $m>3$.

For positive integers $k$ and $m$, let $T_k(m):=\sum_{j=1}^{m} (2j-1)^k$ be the sum of $k$th powers of the first $m$ odd positive integers. A natural question about $T_k(m)$ is whether there are positive integers $k$ and $m>1$ such that $T_k(m+1)/T_k(m)$ is an integer? Since $T_k(m+1)=T_k(m)+(2m+1)^k$, the question can be reformulated as follows: Do there exist positive integers $a$, $k$ and $m>1$ such that ${aT_k(m)=(2m+1)^k}$?

In this paper, we investigate the solvability in positive integers $k$ and $m$ of the equation $aT_k(m)=(2m+1)^k$, where $a$ is a fixed positive integer. In order to do this, we introduce a new type of \emph{helpful pairs}, which provides an important tool for establishing the unsolvability of the equation $aT_k(m)=(2m+1)^k$. The key observation is that for both equations $aS_k(m)=m^k$ and $aT_k(m)=(2m+1)^k$ the same set of helpful pairs can be used. Combining this with some known results on the Kellner--Erd\H{o}s--Moser equation \cite{BM} leads us to our main theorem.
\begin{theorem}
\label{t1}
If $a$ is even or $a$ has a regular prime divisor or $2\le a\le 1500$, then the equation $aT_k(m)=(2m+1)^k$  has no solutions $(m,k)$ with $m>1$.
\end{theorem}

Furthermore, using some known facts about the Erd\H{o}s--Moser equation \cite{K1,MRU1,MRU2}, we deduce our next theorem.

\begin{theorem}
\label{t2}
Let $k$ and $m$ be positive integers satisfying
$$
1^k+3^k+\dots+(2m-1)^k=(2m+1)^k.
$$
Then:
\begin{itemize}
\item[\textup{(a)}]  $2^8\cdot 3^5\cdot 5^4\cdot 7^3\cdot 11^2\cdot 13^2\cdot 17^2\cdot 19^2\cdot\prod_{23\le p<1000} p$ divides $k$.
\item[\textup{(b)}] Every prime divisor of $2m+1$ is greater than $10000$.
\end{itemize}
\end{theorem}

Motivated by Theorems~\ref{t1} and \ref{t2}, we propose the following conjectures.

\begin{conjecture}
$\{T_k(m+1)/T_k(m)\mid k,m\in\mathbb Z^+, m>1\}\cap\mathbb Z=\varnothing$.
\end{conjecture}

\begin{conjecture}
$$
\left\{\left.\frac{S_k(m+1)}{S_k(m)}\,\right | k,m\in\mathbb Z^+, m>3\right\}\cap\mathbb Z=\left\{\left.\frac{T_k(m+1)}{T_k(m)}\,\right | k,m\in\mathbb Z^+, m>1\right\}\cap\mathbb Z.
$$
\end{conjecture}

This paper is organized as follows. In Section~\ref{s2}, we recall some properties of power sums $S_k(m)$ and prove their analogues for $T_k(m)$. In Section~\ref{s3}, we obtain analogues of Theorem~3.11 and Corollaries 3.12 and 3.13 of \cite{BM}. In Section~\ref{s4}, we define helpful pairs of the first and second kind and prove a crucial result (Lemma~\ref{l8}) that allows us to deduce Theorems~\ref{t1} and \ref{t2} as immediate consequences of the results of the previous section and some known facts about the Kellner--Erd\H{o}s--Moser equation. Some other properties of integer solutions of ${aT_k(m)=(2m+1)^k}$ are briefly discussed in Section~\ref{s5}, as well as some related
problems.

\section{Preliminary Lemmas}
\label{s2}

The following result is known as Carlitz-von Staudt's theorem \cite{C, vS} (see also \cite{M3} for a simpler proof).
\begin{lemma}
\label{l1}
Let $k$ and $m$ be positive integers. Then
$$
S_k(m)\equiv\begin{cases}
0\pmod{\frac{m(m-1)}2}&\text{if $k$ is odd},\\
-\sum_{p\mid m,(p-1)\mid k}{\frac mp}\pmod{m} &\text{if $k$ is even}.
\end{cases}
$$
\end{lemma}

\begin{corollary}
\label{c1}
Let $k$, $m$ and $n$ be positive integers and $p>2$ be a prime with $(p-1)\nmid k$. If $m\equiv n\pmod{p}$, then $S_k(m)\equiv S_k(n)\pmod{p}$.
\end{corollary}

\begin{proof}
We may assume that $m>n$. Since $(p-1)\nmid k$, Lemma~\ref{l1} yields $p\mid S_k(m-n)$. Hence $S_k(m)=S_k(m-n)+\sum_{j=0}^{n-1}(m-n+j)^k\equiv S_k(n)\pmod{p}$.
\end{proof}

Our next lemma is the analogue of Lemma~\ref{l1} for $T_k(m)$.
\begin{lemma}
\label{l2}
Let $k$ and $m$ be positive integers. Then
$$
T_k(m)\equiv\begin{cases}
0\pmod{m}&\text{if $k$ is odd},\\
(2^{k-1}-1)\sum_{p\mid (2m+1),(p-1)\mid k}{\frac {2m+1}p}\pmod{2m+1} &\text{if $k$ is even}.
\end{cases}
$$
\end{lemma}

\begin{proof}
First we observe that $T_k(m)=S_k(2m+1)-2^kS_k(m+1)$. Assume that $k$ is odd. By Lemma~\ref{l1}, we have $m(2m+1)\mid S_k(2m+1)$ and $m(m+1)\mid 2^kS_k(m+1)$, and thus $m\mid T_k(m)$. Now assume that $k$ is even. In this case 
$$
S_k(2m+1)=S_k(m+1)+\sum_{j=1}^m (2m+1-j)^k\equiv 2S_k(m+1)\pmod{2m+1}.
$$
Hence $T_k(m)\equiv (1-2^{k-1})S_k(2m+1)\pmod{2m+1}$, and the result follows by Lemma~\ref{l1}.
\end{proof}

The \emph{Bernoulli numbers} $B_0,B_1,B_2,\ldots$ are defined by the generating function
$$
\frac z{e^z-1}=\sum_{j=0}^{\infty} B_j\frac{z^j}{j!},\qquad |z|<2\pi.
$$
They are rational numbers satisfying the recurrence relation
$\sum_{j=0}^k \binom{k+1}jB_j=0$ ($k\ge 1$). It is easy to see that $B_0=1$, $B_1=-1/2$, $B_2=1/6$, $B_4=-1/30$, and $B_j=0$ for all
odd $j>1$. For a positive integer $k$, the $k$th \emph{Bernoulli polynomial} $B_k(x)$ is defined by
$$
B_k(x)=\sum_{j=0}^k \binom kj B_j x^{k-j}.
$$

The following lemma is a special case of Raabe's result~\cite{R}.

\begin{lemma}
\label{l3}
For any positive integer $k$,
$$
B_k(x)=2^{k-1}\left(B_k\left(\frac x2\right)+B_k\Bigl(\frac{x+1}2\Bigr)\right).
$$
\end{lemma}

The next lemma relates power sums to Bernoulli numbers and Bernoulli polynomials (for a proof, see \cite[Chapter~15]{IR}).

\begin{lemma}
\label{l4}
Let $k$ and $m$ be positive integers. Then:
\begin{itemize}
\item[\textup{(a)}]
$S_k(m)=(B_{k+1}(m)-B_{k+1})/(k+1)$.
\item[\textup{(b)}]
$S_k(m)=\sum_{j=0}^k \binom kj B_{k-j}\frac{m^{j+1}}{j+1}$.
\end{itemize}
\end{lemma}

We next obtain the analogue of Lemma~\ref{l4}(b) for $T_k(m)$ with even $k$.

\begin{lemma}
\label{l5}
Let $k$ and $m$ be positive integers, where $k$ is even. Then
$$
T_k(m)=2^k\sum_{j=0}^k \binom kj B_{k-j}\frac{(2m+1)^{j+1}}{2^{j+1}(j+1)}.
$$
\end{lemma}

\begin{proof}
Since $k>0$ is even, we have $B_{k+1}=0$. Then, by Lemmas~\ref{l3} and \ref{l4}(a),
\begin{align*}
T_k(m)&=S_k(2m)-2^kS_k(m)=\frac{B_{k+1}(2m)-2^kB_{k+1}(m)}{k+1}\\
&=\frac{2^k}{k+1}B_{k+1}\Bigl(m+\frac 12\Bigr)=\frac{2^k}{k+1}\sum_{j=0}^{k+1}\binom{k+1}j B_j\cdot \Bigl(m+\frac 12\Bigr)^{k+1-j}\\
&=2^k\sum_{j=0}^k\frac{k!}{j!(k-j+1)!}B_j\cdot\Bigl(m+\frac 12\Bigr)^{k-j+1}=2^k\sum_{j=0}^k\binom kj B_{k-j}\frac{(2m+1)^{j+1}}{2^{j+1}(j+1)}.
\end{align*}
\end{proof}

Write $B_j=U_j/V_j$, where $U_j$ and $V_j$ are integers, $V_j>0$, $\gcd(U_j,V_j)=1$. Kellner~\cite{K2} used Lemma~\ref{l4}(b) to derive the following result (see \cite[Proposition~8.5]{K2}).

\begin{lemma}
\label{l6}
Let $k$ and $m$ be positive integers, where $k$ is even. Then:
\begin{itemize}
\item[\textup{(a)}] $m^2\mid S_k(m)$ if and only if $m\mid U_k$.
\item[\textup{(b)}] $m^3\mid S_k(m)$ if and only if $m^2\mid U_k$.
\end{itemize}
\end{lemma}

The next lemma can be proved in exactly the same way as Lemma~\ref{l6}, except that Lemma~\ref{l5} is invoked instead of Lemma~\ref{l4}(b).
\begin{lemma}
\label{l7}
Let $k$ and $m$ be positive integers, where $k$ is even. Then:
\begin{itemize}
\item[\textup{(a)}] $(2m+1)^2\mid T_k(m)$ if and only if $(2m+1)\mid U_k$.
\item[\textup{(b)}] $(2m+1)^3\mid T_k(m)$ if and only if $(2m+1)^2\mid U_k$.
\end{itemize}
\end{lemma}

An odd prime $p$ is said to be an \emph{irregular prime} if $p$ divides some $U_r$ with even $r\le p-3$. Otherwise, the prime $p$ is said to be a \emph{regular prime}. The pairs $(r,p)$ with $p\mid U_r$ and even $r\le p-3$ are said to be \emph{irregular pairs}.

\section{The Equation $aT_k(m)=(2m+1)^k$}
\label{s3}

First assume that $m=1$. Then $a=3^k$. Next assume that $m>1$ and $k$ is odd. Appealing to Lemma~\ref{l2}, we see that $(2m+1)^k$ must be divisible by $m$, which is impossible. This shows that we may restrict our study to solutions $(m,k)$ with $m>1$ and even $k$.

Proceeding exactly as in Section~3 of \cite{BM} (with $m$ replaced by $2m+1$) and making use of Lemmas~\ref{l2}, \ref{l5} and \ref{l7}, we obtain the following results.

\begin{theorem}
\label{t3}
Suppose that $aT_k(m)=(2m+1)^k$ with $m>1$ and even $k$. Let $p$ be a prime dividing $2m+1$. Then: 
\begin{itemize}
\item[\textup{(a)}] $p$ is an irregular prime.
\item[\textup{(b)}] $k\not\equiv 0,2,4,6,8,10,14\pmod{p-1}$.
\item[\textup{(c)}] $\ord_p (B_k/k)\ge 2\ord_p (2m+1)\ge 2$.
\item[\textup{(d)}] $k\equiv r\pmod{p-1}$ for some irregular pair $(r,p)$.
\end{itemize}
\end{theorem}

\begin{corollary}
\label{c2}
If $a$ has a regular prime divisor, then the equation $aT_k(m)={(2m+1)^k}$ has no solutions with $m>1$.
\end{corollary}

\begin{corollary}
\label{c3}
Let $p_1$ and $p_2$ be distinct irregular prime divisors of $a$. Assume that for every pair $(r_1,p_1)$, $(r_2,p_2)$ of irregular pairs,  $\gcd({p_1-1},p_2-1)\nmid(r_1-r_2)$. Then the equation $aT_k(m)=(2m+1)^k$ has no solutions.
\end{corollary}

\begin{remark}
We observe that $aT_k(m)$ with $k>1$ and $m>1$ can be a perfect $k$th power even when $k$ is odd or $a$ has a regular prime divisor, for example, $315T_2(3)=105^2$ and $12005T_3(5)=245^3$. More examples can be constructed from formulas expressing $T_k(m)$ for small values of~$k$.
\end{remark}

\section{Helpful Pairs}
\label{s4}

For a positive integer $a$ let us call a pair $(t,q)_a$ with $q>3$ a prime and\linebreak $2\le t\le q-3$ even to be a \emph{potentially helpful pair} if $q\nmid a$ and, in case of irregular $q$, $(t,q)$ is not an irregular pair.

Let  $(t,q)_a$ be a potentially helpful pair. We say that $(t,q)_a$ is a \emph{helpful pair of the first kind} if $aS_t(x)\equiv x^t\pmod{q}$ implies $x\equiv 0\pmod{q}$. In view of Corollary~\ref{c1} this definition is equivalent to the definition of a helpful pair given in Section~4 of~\cite{BM}. We say that $(t,q)_a$ is a \emph{helpful pair of the second kind} if $aT_t(x)\equiv (2x+1)^t\pmod{q}$ implies $2x+1\equiv 0\pmod{q}$. The following lemma plays a crucial role in our argument.

\begin{lemma}
\label{l8}
Let  $(t,q)_a$ be a potentially helpful pair. Then $(t,q)_a$ is a helpful pair of the first kind if and only if it is a helpful pair of the second kind.
\end{lemma}

\begin{proof}
Since $(q-1)\nmid t$, we have, by Lemma~\ref{l1}, $q\mid S_t(q)$. Further, as $t$ is even,
$$
S_t(q)=S_t\Bigl(\frac{q+1}2\Bigr)+\sum_{j=1}^{(q-1)/2}(q-j)^t\equiv 2S_t\Bigl(\frac{q+1}2\Bigr)\pmod{q},
$$
and so $q\mid S_t((q+1)/2)$. For a positive integer $x$, we have
$$
S_t\Bigl(x+\frac{q+1}2\Bigr)= S_t\Bigl(\frac{q+1}2\Bigr)+\sum_{j=0}^{x-1}\Bigl(\frac{q+1}2+j\Bigr)^t.
$$
Hence
$$
2^tS_t\Bigl(x+\frac{q+1}2\Bigr)=2^tS_t\Bigl(\frac{q+1}2\Bigr)+\sum_{j=0}^{x-1}(q+1+2j)^t\equiv  T_t(x)\pmod{q}.
$$
This implies
\begin{equation}
\label{eq1}
2^t\left(aS_t\Bigl(x+\frac{q+1}2\Bigr)-\Bigl(x+\frac{q+1}2\Bigr)^t\right)\equiv aT_t(x)-(2x+1)^t\pmod{q},
\end{equation}
and
\begin{equation}
\label{eq2}
2^t\left(aS_t(x)-x^t\right)\equiv aT_t\Bigl(x+\frac{q-1}2\Bigr)-\left(2\Bigl(x+\frac{q-1}2\Bigr)+1\right)^t\pmod{q}.
\end{equation}

Assume that $(t,q)_a$ is a helpful pair of the first kind and for some positive integer $x$ the congruence $aT_k(x)\equiv(2x+1)^t\pmod{q}$ holds. Then, by \eqref{eq1},
$$
aS_t\Bigl(x+\frac{q+1}2\Bigr)\equiv \Bigl(x+\frac{q+1}2\Bigr)^t\pmod{q}.
$$
Hence $x+\frac{q+1}2\equiv 0\pmod{q}$, and so $2x+1\equiv 0\pmod{q}$. Therefore $(t,q)_a$ is a helpful pair of the second kind. 

In a similar manner, making use of \eqref{eq2}, we conclude that if $(t,q)_a$ is a helpful pair of the second kind, then it is a helpful pair of the first kind.
\end{proof}

From now on, we will call both helpful pairs of the first kind and helpful pairs of the second kind simply helpful pairs. Next we prove an analogue of \cite[Lemma~4.4]{BM} for the equation $aT_k(m)=(2m+1)^k$.

\begin{lemma}
\label{l9}
Let $q>3$ be a prime and $2\le t\le q-3$ . If $(t,q)_a$ is a helpful pair and $aT_k(m)=(2m+1)^k$, then $k\not\equiv t\pmod{q-1}$.
\end{lemma}

\begin{proof}
Assume that $k\equiv t\pmod{q-1}$. Then we find that $aT_t(m)\equiv aT_k(m)=(2m+1)^k\equiv(2m+1)^t\pmod{q}$. Hence $2m+1\equiv 0\pmod{q}$. Appealing to parts (a) and (d) of Theorem~\ref{t3}, we conclude that $q$ is irregular and $(t,q)$ is an irregular pair, which conrtadicts the definition of a potentially helpful pair.
\end{proof}

Lemma~\ref{l9} shows that, like in the case of the Kellner--Erd\H{o}s--Moser equation, one can try to use helpful pairs to prove that for a given $a$ and given even positive integers $c$ and $d$ the equation $aT_k(m)=(2m+1)^k$ has no solutions with $k\equiv c\pmod{d}$. Namely, if there is a prime $q>3$ such that $(q-1)\mid d$ and $(t,q)_a$ is a helpful pair, where $t$ is the least nonnegative integer congruent to $c$ modulo $q-1$, then, by Lemma~\ref{l9}, the equation $aT_k(m)=(2m+1)^k$ has no solutions with $k\equiv c\pmod{d}$. Otherwise, we multiply $d$ by an integer $\ell\ge 2$ and consider $\ell$ congruences $k\equiv c+jd\pmod{\ell d}$, $0\le j<\ell$. For each of these congruences we proceed with the same argument as above. Moreover, for both $aS_k(m)=m^k$ and $aT_k(m)=(2m+1)^k$ one can use the \emph{same} set of helpful pairs to rule out the possibility that $k\equiv c\pmod{d}$. This implies that if for a given $a$ it has already been established by means of helpful pairs that the equation $aS_k(m)=m^k$ has no solutions with $k\equiv c\pmod{d}$, then one can conclude that the equation $aT_k(m)=(2m+1)^k$ has no solutions with $k\equiv c\pmod{d}$, and vice versa. In view of this, we have the following analogue of Proposition~7.1 of \cite{BM}.

\begin{proposition}
\label{p1}
Suppose that $aT_k(m)=(2m+1)^k$ with $m>1$. Then:
\begin{itemize}
\item[\textup{(a)}] If $a\equiv 1$ or $2$ or $3\pmod{5}$, then $4\mid{k}$.
\item[\textup{(b)}] If $a\equiv 1$ or $3$ or $5\pmod{7}$, then $6\mid{k}$.
\item[\textup{(c)}] If $a\equiv 6$ or $7\pmod{11}$, then $10\mid{k}$.
\item[\textup{(d)}] If $a\equiv 2$ or $8$ or $11\pmod{13}$, then $12\mid{k}$.
\item[\textup{(e)}] If $a\equiv 1$ or $6\pmod{13}$, then $6\mid k$.
\item[\textup{(f)}] If $a\equiv 1$ or $5\pmod{11}$ and $a\equiv 15\pmod{31}$, then $10\mid{k}$.
\end{itemize}
\end{proposition}

Furthermore, Proposition~4.5 of \cite{BM} can be extended as follows.

\begin{proposition}
\label{p2}
Let $p$ be an irregular prime dividing $a$. Assume that for every irregular pair $(r,p)$ there exists a positive integer $\ell_r$ such that for every\linebreak $j=0,1,\dots,\ell_r-1$ there is a helpful pair $(t_j,q_j)_a$ with $(q_j-1)\mid \ell_r(p-1)$ and $t_j\equiv r+j(p-1)\pmod{q_j-1}$. Then both $aS_k(m)=m^k$ and $aT_k(m)=(2m+1)^k$ have no solutions.
\end{proposition}

The procedure described in Proposition~\ref{p2} has recently been used to show \cite{BM} that if $2\le a\le 1500$ and the prime divisors of $a$ are all irregular, then the equation $aS_k(m)=m^k$ has no solutions (the corresponding helpful pairs are listed in Table~3 of \cite{BM}). As an immediate consequence of this, we obtain the following.

\begin{proposition}
\label{p3}
If $2\le a\le 1500$ and all prime divisors of $a$ are irregular, then the equation $aT_k(m)=(2m+1)^k$ has no solutions.
\end{proposition}

Combining Proposition~\ref{p3} with Corollary~\ref{c2} and an obvious fact that for even $a$ the equation $aT_k(m)=(2m+1)^k$ has no solutions, we deduce Theorem~\ref{t1}.

If an irregular prime $p$ divides $a$, then we have a reasonably good chance to establish the unsolvability of the equations $aS_k(m)=m^k$ and $aT_k(m)=(2m+1)^k$ by applying Proposition~\ref{p2}. Unfortunately, it is not the case when $a=1$. However, helpful pairs can still be used to develop some divisibility properties of positive integers $k$ and $m$ satisfying either $S_k(m)=m^k$ or $T_k(m)=(2m+1)^k$. For the original Erd\H{o}s--Moser equation $S_k(m)=m^k$, Moree, te Riele and Urbanowicz~\cite{MRU2} have proved by means of helpful pairs that $2^8\cdot 3^5\cdot 5^4\cdot 7^3\cdot 11^2\cdot 13^2\cdot 17^2\cdot 19^2\cdot\prod_{23\le p< 200} p$\linebreak divides $k$ and every prime divisor of $m$ is greater than $10000$, provided that $k>1$ (see also \cite{MRU1} for more details). Kellner~\cite{K1} has shown that also all primes ${200<p<1000}$ divide $k$. Combining these results with the above remarks, we deduce Theorem~\ref{t2}.

\begin{remark}
Although the theoretical results of \cite{MRU1} and \cite{MRU2} are stated and proved in terms of good pairs (of which the helpful pairs are a special case), the numerical results mentioned above have been obtained  using only helpful pairs.
\end{remark}

\section{Concluding Remarks}
\label{s5}
It is not difficult to find some other similarities between the Kellner--Erd\H{o}s--Moser equation $aS_k(m)=m^k$ and the equation $aT_k(m)=(2m+1)^k$. For example, using Lemma~10.2 of \cite{BM} and the relation $T_k(m)=S_k(2m)-2^kS_k(m)$ one can readily deduce the following result (cf. \cite[Theorem~10.1]{BM}).
\begin{theorem}
\label{t4}
Suppose that $aT_k(m)=(2m+1)^k$ with $m>1$. Then
$$
\ord_2(am-1)\begin{cases}
=2+\ord_2 k & \text{if $a\equiv 1\!\!\pmod{4}$},\\
\ge 3+\ord_2 k & \text{if $a\equiv 3\!\!\pmod{4}$}.
\end{cases}
$$
\end{theorem}
Furthermore, one can easily show that if $(m_1,k_1)$ and $(m_2,k_2)$ are two distinct solutions of $aT_k(m)=(2m+1)^k$, then $m_1\ne m_2$ and $k_1\ne k_2$ (cf. \cite[Proposition~1.7]{BM}).

By employing the same type of argument as in \cite{BM} and in Sections~\ref{s3} and \ref{s4} of this paper, one can derive some divisibility properties of integers $k$ and $m$ satisfying either $S_k(m)=am^k$ or  $T_k(m)=a(2m+1)^k$, where $a$ is a fixed positive integer (see \cite{M2} and \cite{M4} for some results on $S_k(m)=am^k$). In particular, we are able to prove analogues of Theorem~\ref{t3} and Proposition~\ref{p1}, however, so far we cannot make any conclusions about the unsolvability of these equations.

\section*{Acknowledgement}
The author thanks an anonymous referee for helpful suggestions.


\begin{thebibliography}{00}

\bibitem{BM}
I.~N.~Baoulina, P.~Moree, Forbidden integer ratios of consecutive power sums, in: From Arithmetic to Zeta-Functions, Springer International Publishing, 2016, pp.~1--30.

\bibitem{C}
L.~Carlitz, The Staudt-Clausen theorem, Math. Mag. 34~(1960/1961) 131--146.

\bibitem{GMZ} 
Y.~Gallot, P.~Moree, W.~Zudilin, The Erd\H{o}s-Moser equation
$1^k+2^k+\dots+(m-1)^k=m^k$ revisited using continued fractions, Math. Comp. 80~(2011) 1221--1237.

\bibitem{IR}
K.~Ireland, M.~Rosen, A classical introduction to modern number theory, Springer-Verlag, New York, 1990.

\bibitem{K1} 
B.~C.~Kellner, \"Uber irregul\"are Paare h\"oherer Ordnungen,
Diplomarbeit, Mathematisches Institut der Georg--August--Universit\"at zu G\"ottingen, 
Germany, 2002. (Also available at \texttt{http://www.bernoulli.org/\~{}bk/irrpairord.pdf})

\bibitem{K2}
B.~C.~Kellner, On irregular prime power divisors of the Bernoulli numbers, Math. Comp. 76~(2007) 405--441.

\bibitem{K3} 
B.~C.~Kellner, On stronger conjectures that imply the Erd\H{o}s-Moser
conjecture, J.~Number Theory 131~(2011) 1054--1061.

\bibitem{M2} 
P.~Moree, Diophantine equations of Erd\H{o}s-Moser type, Bull. Austral. Math. Soc.  53~(1996) 281--292.

\bibitem{M3} 
P.~Moree, A top hat for Moser's four mathemagical rabbits, Amer. Math. Monthly  118~(2011) 364--370.

\bibitem{M4} 
P.~Moree, Moser's mathemagical work on the equation
$1^k+2^k+\cdots+(m-1)^k=m^k$, Rocky Mountain J. Math. 43~(2013) 1707--1737.

\bibitem{MRU1}
P.~Moree, H.~J.~J.~te~Riele, J.~Urbanowicz, Divisibility properties of integers $x$ and $k$
satisfying $1^k+2^k+\dots+(x-1)^k=x^k$, Report NM-R9215, Centrum voor Wiskunde en
Informatica, Amsterdam, August 1992.

\bibitem{MRU2}  
P.~Moree, H.~J.~J.~te Riele, J.~Urbanowicz, Divisibility 
properties of integers $x$, $k$ satisfying 
$1^k+\dots+(x-1)^k=x^k$, Math. Comp. 63~ (1994)  799--815.

\bibitem{Moser} 
L.~Moser, On the 
diophantine equation $1^n+2^n+3^n+\dots +(m-1)^n=m^n$, Scripta Math. 19~(1953) 84--88.

\bibitem{R}
J.~L.~Raabe, Zur\"uckf\"uhrung einiger Summen und bestimmten Integrale auf die Jacob Bernoullische
Function, J.~Reine Angew. Math. 42~(1851) 348--376.

\bibitem{vS}
K.~G.~C.~von Staudt, Beweis eines Lehrsatzes, die Bernoullischen Zahlen betreffend, J.~Reine Angew. Math. 21~(1840) 372--374.

\end{thebibliography}
\end{document}